\documentclass[12pt]{amsart}
\usepackage[margin=1in]{geometry}
\usepackage{amscd, amssymb, amsmath, wasysym}
\usepackage{graphicx}
\usepackage{amsfonts}
\usepackage{mathrsfs}    
\usepackage{eucal}     
\usepackage{latexsym}   
\usepackage{verbatim}   
\usepackage[all]{xy}   
\usepackage[dvipsnames]{xcolor}
\usepackage{bookmark}
\usepackage{subcaption}\usepackage{BOONDOX-calo}

 \usepackage{hyperref}
 \hypersetup{
     colorlinks=true,
     linkcolor=NavyBlue,
     filecolor=NavyBlue,
     citecolor = TealBlue,
     urlcolor=magenta,
  }

\theoremstyle{definition}

\theoremstyle{remark}

\theoremstyle{corollary}

\newtheorem*{conjectureBI}{Barth--Ionescu Conjecture}
\newtheorem*{conjectureCI}{Hartshorne Conjecture}

\theoremstyle{theorem}

\theoremstyle{corollary}

\newtheorem{theorem}{Theorem}[section]
\newtheorem{lemma}[theorem]{Lemma}
\newtheorem{proposition}[theorem]{Proposition}

\theoremstyle{corollary}
\newtheorem{corollary}[theorem]{Corollary}

\theoremstyle{definition}

\theoremstyle{remark}
\newtheorem{remark}[theorem]{Remark}

\numberwithin{equation}{section}

\newcommand{\sO}{\mathscr{O}}
\newcommand{\sI}{\mathscr{I}}
\def\nP{\mathbf{P}}

\def\Pic{\operatorname{Pic}}

\def\reg{\operatorname{reg}}

\newcommand{\suchthat}{\;\ifnum\currentgrouptype=16 \middle\fi|\;}

\input xy
\xyoption{all}

\title[Smooth projective varieties of small degree and codimension]{Some remarks on smooth projective varieties \\of small degree and codimension}

\begin{document}

\author{Jinhyung Park}
\address{Department of Mathematical Sciences, KAIST, 291 Daehak-ro, Yuseong-gu, Daejeon 34141, Republic of Korea}
\email{parkjh13@kaist.ac.kr}

\date{\today}
\subjclass[2020]{14N25, 14E25}
\keywords{small degree, small codimension, Hartshorne conjecture, projection, Fano variety}

\thanks{The author was partially supported by the National Research Foundation (NRF) funded by the Korea government (MSIT) (NRF-2022M3C1C8094326).}


\begin{abstract}
The purpose of this note is twofold. First, we give a quick proof of Ballico--Chiantini's theorem stating that a Fano or Calabi--Yau variety of dimension at least $4$ in codimension two is a complete intersection. Second, we improve Barth--Van de Ven's result asserting that if the degree of a smooth projective variety of dimension $n$ is less than approximately $0.63 \cdot n^{1/2}$, then it is a complete intersection. We show that the degree bound can be improved to approximately $0.79 \cdot n^{2/3}$.
\end{abstract}

\maketitle


\section{Introduction}

\noindent Throughout the note, $X \subseteq \nP^r$ is a non-degenerate smooth projective complex variety of dimension $n$, codimension $e=r-n$, and degree $d$, and $H$ is its hyperplane section. In 1970's, Barth and Hartshorne independently observed that if $e$ is small compared with $n$, then $X$ is homologically like a complete intersection. The fact that no projective manifolds in small codimension were known other than complete intersections led Hartshorne to make the following famous conjecture.

\begin{conjectureCI}[{\cite{Hartshorne}}]
If $n \geq 2e+1$ or $n=4$ and $e=2$, then $X \subseteq \nP^r$ is a complete intersection.
\end{conjectureCI}

This conjecture is widely open even when $e=2$ or $X$ is a Fano variety. However, based on Ran's deep result \cite{Ran}, Ballico--Chiantini \cite{BC} verified Hartshorne conjecture when $e=2$ and $X$ is a Fano variety (or a Calabi--Yau variety).

\begin{theorem}[{\cite[Theorem 6]{BC}}]\label{thm:BC}
If $X$ is a Fano or Calabi--Yau variety, then $X \subseteq \nP^{n+2}$ is a complete intersection. In particular, if $d \leq 2n+2$, then $X \subseteq \nP^{n+2}$ is a complete intersection.
\end{theorem}

In this note, we give a quick proof of the theorem. We do not use Ran's result \cite{Ran}. Instead, we exploit the projection method developed by Lazarsfeld \cite{Lazarsfeld} for the Castelnuovo--Mumford regularity (see also \cite{Kwak}). Consider a general projection $\pi \colon X \rightarrow \overline{X} \subseteq \nP^{n+1}$. We use Ballico--Chiantini's argument (this is the only part where we use their argument) to prove that $X$ is contained in a hypersurface of degree $n$ (Lemma \ref{lem:h^0(I_X(n))}). This yields a surjective map
$$
\sO_{\nP^{n+1}}(-n+1) \oplus \cdots \oplus \sO_{\nP^{n+1}}(-1) \oplus \sO_{\nP^{n+1}} \longrightarrow \pi_{*} \sO_X
$$
whose kernel $E$ is a vector bundle of rank $n$ such that $H^i(\nP^{n+1}, E^{\vee}(m))=0$ for $1 \leq i \leq n-1$ and $m \in \mathbf{Z}$. By Evans--Griffith theorem \cite{EG}, $E$ is splitting. This implies that $X \subseteq \nP^{n+2}$ is projectively normal, and it is well-known that it suffices to conclude the theorem.

\medskip

Now, recall a famous theorem of Hartshorne \cite{Hartshorne} that there exists a function $\mathcal{N}(n)$ such that if $d \leq \mathcal{N}(n)$, then $X \subseteq \nP^r$ is a complete intersection. Subsequently, Barth--Van de Ven \cite{Barth} proved that one can take
$$
\mathcal{N}(n) = \sqrt{\frac{2}{5}n + \frac{1}{4}} + \frac{1}{2}.
$$
Note that the Segre embedding of $\nP^{n-1} \times \nP^1$ into $\nP^{2n-1}$ is not a complete intersection but has dimension $n$ and degree $n$. Moreover, the Pl\"{u}cker embedding of the Grassmannian $\mathbf{G}(1,4)$ into $\nP^9$ is not a complete intersection but has dimension $6$ and degree $5$. Considering these examples, the following conjecture seems the best possible.

\begin{conjectureBI}[{cf. \cite[Conjecture 3.5]{IR2}}]
If $d \leq n-1$, then $X \subseteq \nP^r$ is a complete intersection unless it is the Pl\"{u}cker embedding $\mathbf{G}(1,4) \subseteq \nP^9$.
\end{conjectureBI}

This conjecture suggests that one can take $\mathcal{N}(n)=n-2$.  In this note, we show that one can take
$$
\mathcal{N}(n)=\min \left\{n-2, \sqrt[3]{\frac{(n+2)^2}{2}}+2 \right\}.
$$
When $n \gg 0$, aforementioned Barth--Van de Ven's bound is approximately $0.63 \cdot n^{1/2}$, but our bound is approximately $0.79 \cdot n^{2/3}$. It is elementary to see that if $d \leq n-1$, then $n \geq 2e+1$ unless $X \subseteq \nP^r$ is the Pl\"{u}cker embedding $\mathbf{G}(1,4) \subseteq \nP^9$ (Corollary \ref{cor:d<=n-2}). Thus Barth--Ionescu conjecture is a special case of Hartshorne conjecture. Moreover, our assertion on $\mathcal{N}(n)$ raised above is a consequence of the following theorem.

\begin{theorem}\label{thm:main}
Assume that $n \geq 2e+1$. If 
$$
d \leq \sqrt[3]{\frac{(n+e^2-2e+2)^2}{2}}+2,
$$
then $X \subseteq \nP^r$ is a complete intersection.
\end{theorem}

Our approach to proving the theorem is completely different from Barth--Van de Ven's approach which is based on the geometry of the varieties of lines contained in $X$. We first establish that if $d \leq (n+e)/e+e$, then $X \subseteq \nP^r$ is a complete intersection except for a few trivial cases of $d=e+1$ and $d=e+2$ (Theorem \ref{thm:(n+2)/e+2}). This is an improvement of Bertram--Ein--Lazarsfeld's result \cite[Corollary 3]{BEL} which is the same assertion without the exceptional cases but replaces the degree condition with $d \leq (n+e)/2e$. For proving Theorem \ref{thm:(n+2)/e+2}, we basically follow the proof of \cite[Corollary 3]{BEL}, but instead of applying the main result of \cite{BEL}, we use Noma's theorem \cite[Theorem 9]{Noma} to show that $X \subseteq \nP^r$ is arithmetically Cohen--Macaulay. We then argue that the defining ideal $I_{X|\nP^r}$ is generated by $(e-1)(d-1-e/2)$ generators (Lemma \ref{lem:numbgens}). As $(e-1)(d-1-e/2) \leq n$ and we may assume $n \geq 3e-2$, Netsvetaev's theorem \cite{Netsvetaev} yields that  $X \subseteq \nP^r$ is a complete intersection. To deduce Theorem \ref{thm:main}, we establish that if $n \geq 2e+1$, then $d \geq e^2/2 + 2$ (Lemma \ref{lem:d<=e(e-1)/2+2}). This result might be of independent interest. We refer to Section \ref{sec:mindeg} for other results on degree of a smooth projective variety of small codimensio). Now, the degree condition in Theorem \ref{thm:main} implies $d \leq (n+e)/e+e$, so $X \subseteq \nP^r$ is a complete intersection by Theorem \ref{thm:(n+2)/e+2}.

\medskip

The organization of the note is as follows. After discussing some properties of the degree of a smooth projective variety of small codimension using a general projection in Section \ref{sec:mindeg}, we prove Theorem \ref{thm:BC} in Section \ref{sec:proof1} and Theorem \ref{thm:main} in Section \ref{sec:proof2}.

\section{Degree of a smooth projective variety of small codimension}\label{sec:mindeg}

\subsection{Remarks on degree of Fano manifolds}\label{subsec:Fano}  In this subsection, we suppose the following:\\[3pt]
\indent $(a)$ $H^i(X, \sO_X(m))=0$ for $1 \leq i \leq n-1$ and $m \in \mathbf{Z}$.\\[3pt]
\indent $(b)$ $K_X=kH$ for an integer $k$ (i.e., $X \subseteq \nP^r$ is subcanonical).\\[3pt]
\indent $(c)$ $X \subseteq \nP^r$ is linearly normal.\\[3pt]
When $X$ is Fano, the condition $(a)$ holds by Kodaira vanishing theorem. When $n \geq e+2$, the condition $(b)$ holds by Barth--Larsen theorem (cf. \cite[Corollary 3.2.3]{positivity}). When $n \geq 2e-1$, the condition $(c)$ holds by Zak's theorem on linear normality (cf. \cite[Theorem 3.4.25]{positivity}).

\medskip

Now, take a general projection
$$
\pi \colon X \longrightarrow \nP^n
$$
centered at a general $(e-1)$-dimensional linear subspace $\nP^r$ disjoint from $X$. Note that $\pi$ is a finite flat morphism. Then $\pi_{*}\sO_X$ is a vector bundle of rank $d$ on $\nP^n$. By Horrocks criterion, the condition $(a)$ implies that $\pi_{*}\sO_X$ is splitting. The duality for a finite flat morphism and the condition $(b)$ show that
$$
(\pi_{*}\sO_X)^{\vee} = \pi_{*} \omega_{X/\nP^n} = \pi_{*} \sO_X(n+1+k).
$$
Thus we may write
$$
\pi_{*}\sO_X = \sO_{\nP^n}^{\oplus a_0} \oplus \sO_{\nP^n}(-1)^{\oplus a_1} \oplus \cdots \oplus \sO_{\nP^n}(-(n+1+k))^{\oplus a_{n+1+k}},
$$
where $a_0, a_1, \ldots, a_{n+1+k}$ are nonnegative integers such that 
$$
\text{$a_0=1, a_1=e$, $a_i = a_{n+1+k-i}$ for all $0 \leq i \leq n+1+k$, and $a_0 + a_1 + \cdots + a_{n+1+k} = d$.}
$$
Here we use the condition $(c)$ to get $a_1=e$.
As $n+1+k \geq 0$, we get $k \geq -n-1$. The following are immediate:\\[3pt]
$(1)$ $k=-n-1 \Longleftrightarrow d=1$. In this case, $X=\nP^n$.\\[3pt]
$(2)$ $k=-n  \Longleftrightarrow d=2$. In this case, $X$ is a quadric hypersurface in $\nP^{n+1}$.\\[3pt]
$(3)$ $k=-n+1  \Longleftrightarrow d=e+2$. In this case, $X$ is a del Pezzo manifold when $n \geq 2$.\\[3pt]
$(4)$ $k=-n+2  \Longleftrightarrow d=2e+2$. In this case, $X$ is a Mukai manifold when $n \geq 3$.\\[3pt]
Del Pezzo manifolds (resp. Mukai manifolds) are classified by Fujita \cite{Fujita1, Fujita2} (resp. Mukai \cite{Mukai}). If $-K_X = (n-3)H$ (resp. $-K_X=(n-4)H$), then $d=2e+a_2 +2$ (resp. $d=2e+2a_2+2$). It would be interesting to find all possible values of $e$ and $a_2$. The following are partial results.

\begin{lemma}\label{lem:k>=-n+e=>3e+2}
If $k \geq -n+3$, then $d \geq 3e+2$. In particular, if $k=-n+3$, then $a_2 \geq e$.
\end{lemma}

\begin{proof}
We have $d \geq 2e+2+a_2$. Take a general curve section $C \subseteq \nP^{e+1}$, and let $g$ be the genus of $C$. Then $K_C = (k+n-1)H|_C$ with $k+n-1 \geq 2$, so $2g-2 \geq 2d$. Suppose that $d \leq 3e+1$. By Castelnuovo's genus bound (see e.g., \cite{Harris}), $g \leq 2d-3e-2$.
As $d +1 \leq g$, we get $3e+3 \leq d$, which is a contradiction. Thus $d \geq 3e+2$.
\end{proof}

\begin{lemma}\label{lem:-K_X=(n-3)He=2}
Suppose that $-K_X=(n-3)H$ with $n \geq 2$. If $e=2$, then $a_2=2$ or $3$.
\end{lemma}

\begin{proof}
We have $d=a_2+6$. Take a general surface section $S \subseteq \nP^4$. Then $K_S = H|_S=:H_S$. Recall the self-intersection formula:
$d^2 - 10d - 5 H_S.K_S - 2K_S^2 + 12 + 12p_a(S) = 0$.
As $p_a(S)=p_g(S)=h^0(S, K_S) = 5$, we have
$$
d^2-17d + 72 = (d-8)(d-9) = 0.
$$
Thus $d=8$ ($a_2=2$) or $d=9$ ($a_2=3$). 
\end{proof}

\begin{proposition}\label{prop:-K_X=(n-3)H=>d<=39}
Suppose that $-K_X = (n-3)H$ with $n \geq 2e-1$ and $e \geq 3$. If $n \geq 10$, then $e \leq 6$ and $d \leq 4e+15 \leq 39$.
\end{proposition}

\begin{proof}
By Barth--Larsen theorem (cf. \cite[Corollary 3.2.3]{positivity}), $\Pic(X)$ is generated by $H$. As $n-3 > 2n/3$, \cite[Proposition 4.3]{IR1} implies that $X$ is a conic-connected manifold. Then \cite[Proposition 3.2]{IR1} shows that 
$$
n-3 \leq \frac{n+\delta(X)}{2}=\frac{2n-e+1}{2},
$$
where $\delta(X)=(2n+1)-(n+e)$ is the secant defect which can be computed by Zak's theorem on linear normality (cf. \cite[Theorem 3.4.26]{positivity}). Thus $e \leq 7$. If $e=7$, then $n-3 = (n+\delta(X))/2$. By \cite[Proposition 3.2]{IR1}, $X$ is a local quadratic entry locus manifold (LQELM) of type $\delta(X)$, but it is impossible by Russo's classification \cite[Corollary 3.1]{Russo}. Thus $e \leq 6$. Now, recall from \cite[Lemma 3.1]{Floris} that
$$
n+e+1=h^0(X, H) = \frac{d}{24}(-n^2+7n-8) + \frac{c_2(X) H^{n-2}}{12} + n-3.
$$
Since $n-3 > (n+1)/2$, it follows that $T_X$ is $H$-semistable. Applying the Bogomolov inequality as in \cite[Proof of Proposition 3.3]{Floris}, we get
\begin{equation}\label{eq:d<4e+16}
e+4 \geq \frac{7n-9}{24n} d > \frac{1}{4}d.
\end{equation}
Thus $d < 4e+16$, so $d \leq 4e+15 \leq 39$.
\end{proof}


\subsection{Smallest possbile degree}
As $X \subseteq \nP^r$ is non-degenerate, we have $d \geq e+1$. We say $X \subseteq \nP^r$ is a \emph{variety of minimal degree} if $d=e+1$. A famous classical theorem of del Pezzo and Bertini says that a variety of minimal degree is either a rational normal scroll or the second Veronse surface $v_2(\nP^2) \subseteq \nP^5$. Observe that if $d = e+1$, then $e \geq n-1$. It is natural to wonder what the smallest degree is among all non-degenerate smooth projective varieties of dimension $n$ in $\nP^{n+e}$ with $e \leq n-2$.  The following is a first step to solving this problem.

\begin{proposition}\label{prop:mindeg}
Assume that $2 \leq e \leq n-2$. We have the following:\\[3pt]
$(1)$ $d \geq e+2$, and $d=e+2$ if and only if $X \subseteq \nP^r$ is a complete intersection of type $(2,2)$, or the Pl\"{u}cker embedding $\mathbf{G}(1,4) \subseteq \nP^9$ or its hyperplane section $(n=5, 6, e=3$).\\[3pt]
$(2)$ Suppose that $d \geq e+3$. Then $d \geq 2e+2$, and $d=2e+2$ if and only if $X \subseteq \nP^r$ is a complete intersection of type $(2,3)$ or $(2,2,2)$, a hyperquadric section of the cone of $\mathbf{G}(1,4) \subseteq \nP^9$ $(n=6, e=4)$, the Pl\"{u}cker embedding $\mathbf{G}(1,5) \subseteq \nP^{14}$ $(n=8, e=6)$, or the spinor variety in $\nP^{15}$ or its linear sections $(n=7,8,9,10, e=5)$.\\[3pt]
$(3)$ Suppose that $d \geq 2e+3$. Then $d \geq 3e+2$ unless $X \subseteq \nP^r$ is not an isomorphic projection of the Pl\"{u}cker embedding $\mathbf{G}(1,5) \subseteq \nP^{14}$.
\end{proposition}

\begin{proof}
By Barth--Larsen theorem (cf. \cite[Corollary 3.2.3]{positivity}), $\Pic(X)$ is generated by $H$. If $X$ is not Fano, then \cite[Theorem 1.2]{KP} says that $d \geq ne+2 \geq 4e+2$. Thus we may assume that $X$ is Fano, so we may write $-K_X = \ell H$ for an integer $1 \leq \ell \leq n-1$. Recall that $d=e+2$ if and only if $\ell=n-1$ and $X \subseteq \nP^r$ is linearly normal. Then $(1)$ follows from Fujita's classification \cite{Fujita1, Fujita2}. Suppose that $d \geq e+3$. By Zak's theorem (cf. \cite[Theorem 3.4.26]{positivity}), all varieties in $(1)$ do not have an isomorphic projection. Thus $\ell \leq n-2$, so $d \geq 2e+2$. Recall that $d=2e+2$ if and only if $\ell = n-2$ and $X \subseteq \nP^r$ is linearly normal. Then $(2)$ follows from Mukai's classification \cite{Mukai}. Suppose that $d \geq 2e+3$. By Zak's theorem (cf. \cite[Theorem 3.4.26 and Remark 3.4.27]{positivity}), all varieties in $(2)$ do not have an isomorphic projection except the Pl\"{u}cker embedding $\mathbf{G}(1,5) \subseteq \nP^{14}$. Thus we may assume that $\ell \leq n-3$. Then $(3)$ follows from Lemma \ref{lem:k>=-n+e=>3e+2}.
\end{proof}

\begin{corollary}\label{cor:d<=n-2}
If $d \leq n-1$, then $d \geq 2e+2$ $($hence $n \geq 2e+3$$)$ or $X \subseteq \nP^r$ is a complete intersection of type $(2,2)$ or the Pl\"{u}cker embedding $\mathbf{G}(1,4) \subseteq \nP^9$.
\end{corollary}

\begin{proof}
As $e+1 \leq d \leq n-1$, we get $e \leq n-2$. By Proposition \ref{prop:mindeg}, $d \geq 2e+2$ unless $X \subseteq \nP^r$ is a complete intersection of type $(2,2)$ or the Pl\"{u}cker embedding $\mathbf{G}(1,4) \subseteq \nP^9$. If $d \geq 2e+2$, then $2e+2 \leq d \leq n-1$, so $n \geq 2e+3$. 
\end{proof}

\begin{lemma}\label{lem:d<=e(e-1)/2+2}
Assume that $n \geq 2e+1$. Then 
$$
d \geq \frac{e^2}{2}+2.
$$
\end{lemma}

\begin{proof}
We may assume that $e \geq 2$.
Suppose that $d < e^2/2+2$. As $e^2/2 + 2 \leq ne+1$, \cite[Theorem 1.2]{KP} shows that $X$ is a Fano variety so that we may write $-K_X = \ell H$ for some integer $1 \leq \ell \leq n+1$. Let $k$ be an integer such that $ke +2 \leq d \leq (k+1)e + 1$. Notice that $k<e/2$. Take a general linear section $Y \subseteq \nP^{k+1+e}$ of dimension $k+1$. We have $K_Y = (n-\ell-k-1)H|_Y$. By \cite[Theorem 1.2]{KP} (see also \cite[p.44]{Harris}), $H^0(Y, K_Y)=0$, so $n-\ell-k-1 < 0$, i.e., $\ell \geq n-k$. If $(2/3)n \geq n-k$ (i.e., $k \geq n/3$), then
$$
\frac{ne}{3} +2 \leq ke+2 \leq d < \frac{e^2}{2}+2 \leq \frac{2e^2}{3}+2.
$$
Thus $n \leq 2e$, which is a contradiction. Hence $\ell \geq n-k > (2/3)n$. By \cite[Proposition 4.3]{IR1}, $X$ is a conic-connected manifold, and then by \cite[Proposition 3.2]{IR1}, we have
$$
n-k \leq \ell \leq \frac{n+\delta(X)}{2}= \frac{2n-e+1}{2},
$$
where $\delta(X) = (2n+1)-(n+e)=n-e+1$ is the secant defect, which can be computed by Zak's theorem on linear normality (cf. \cite[Theorem 3.4.26]{positivity}). Suppose that $\ell  = (2n-e+1)/2$. Then \cite[Proposition 3.2]{IR1} says $X$ is a local quadratic entry locus manifold (LQELM) of type $\delta(X)$. Note that $n/2 < \delta(X) < n$. By Russo's classification \cite[Corollary 3.1]{Russo}, it is impossible to satisfy $n \geq 2e+1$. Thus 
$$
n-k \leq \ell \leq \frac{2n-e}{2}.
$$
This implies that $e/2 \leq k$, which is a contradiction.
\end{proof}

\section{Proof of Theorem \ref{thm:BC}}\label{sec:proof1}

\noindent  In this section, we prove Theorem \ref{thm:BC}. We assume $e=2$ and $n \geq 4$.

\begin{lemma}\label{lem:projmethod}
Suppose that $H^i(X, \sO_X(m))=0$ for $1 \leq i \leq n-1$ and $m \in \mathbf{Z}$ and $\operatorname{length}(\langle P, x \rangle \cap X) \leq n$ for all $x \in X$ and a fixed general point $P \in \nP^{n+2}$. Then $X \subseteq \nP^{n+2}$ is a complete intersection.
\end{lemma}

\begin{proof}
We take a general projection
$$
\pi \colon X \longrightarrow \overline{X} \subseteq \nP^{n+1}
$$
centered the fixed general point $P \in \nP^{n+2}$. Then there is a surjective map
$$
\nu \colon \sO_{\nP^{n+1}}(-(n-1)) \oplus \cdots \oplus \sO_{\nP^{n+1}}(-1) \oplus \sO_{\nP^{n+1}} \longrightarrow \pi_{*} \sO_X
$$
such that $E:=\ker(\nu)$ is a vector bundle of rank $n$ on $\nP^{n+1}$ (see \cite{Kwak} and \cite{Lazarsfeld}). 
Observe that
$$
H^i(\nP^{n+1}, E^{\vee}(m)) = H^{n+1-i}(\nP^{n+1}, E(-m-n-2))=0~~\text{ for $1 \leq i \leq n-1$ and $m \in \mathbf{Z}$}.
$$
By Evans--Griffith theorem \cite{EG} (see \cite[Proposition 1.4]{Ein1}), $E^{\vee}$ is splitting, and so is $E$. Then\\[-25pt]

\begin{small}
$$
H^0(\nP^{n+1}, \sO_{\nP^{n+1}}(m-(n-1)) \oplus \cdots \oplus \sO_{\nP^{n+1}}(m-1)  \oplus \sO_{\nP^{n+1}}(m)) \xrightarrow{H^0(\nu \otimes \sO_{\nP^{n+1}}(m))} H^0(X, \sO_X(m))
$$
\end{small}

\noindent is surjective for any $m \geq 0$, so $X \subseteq \nP^{n+2}$ is projectively normal. By Evans--Griffith theorem \cite{EG} (see \cite[Corollary 1.5]{Ein1}), $X \subseteq \nP^{n+2}$ is a complete intersection.
\end{proof}

\begin{remark}
Let $X \subseteq \nP^{n+3}$ be a non-degenerate smooth projective Fano variety of dimension $n \geq 7$. Suppose that there exists a positive integer $k$ with $(k+2)(k+1)/2 \leq n$ such that  $\operatorname{length}(\langle L \cap x \rangle \cap X) \leq k+1$ (or $\reg(\langle L \cap x \rangle \cap X) \leq k+1$) for all $x \in X$ and a fixed general line $L \subseteq \nP^{n+3}$. By the arguments in the proof of Lemma \ref{lem:projmethod}, $X \subseteq \nP^{n+2}$ is projectively normal. Then \cite[Proposition 1.4]{Ein2} implies that $X \subseteq \nP^{n+3}$ is a complete intersection.
\end{remark}

Now, we further suppose that $X$ is Fano or Calabi--Yau. By Kodaira vanishing and Serre duality, $H^i(X, \sO_X(m))=0$ for $1 \leq i \leq n-1$ and $m \in \mathbf{Z}$. To get $H^i(X, \sO_X)=0$ for $1 \leq i \leq n-1$ when $X$ is Calabi--Yau, we also need to apply Barth--Larsen theorem (cf. \cite[Theorem 3.2.1]{positivity}). Our proof of the following lemma is inspired by \cite[Proof of Theorem 6]{BC} but much simpler.

\begin{lemma}\label{lem:h^0(I_X(n))}
Suppose that $X$ is Fano or Calabi--Yau so that $-K_X=\ell H$ for $0 \leq \ell \leq n-1$. Then $H^0(\nP^{n+2}, \sI_{X|\nP^{n+2}}(n)) \neq 0$.
\end{lemma}

\begin{proof}
Recall from Subsection \ref{subsec:Fano} that $d=a_2+6$.
Suppose $\ell \geq n-2$. Then $\ell = n-1$ or $n-2$, so $a_2 \leq 2$. We find
$$
\begin{array}{rcl}
h^0(X, \sO_X(2)) &=& h^0(\nP^n, \sO_{\nP^n}(2)) + 2 h^0(\nP^n, \sO_{\nP^n}(1)) + a_2 h^0(\nP^n, \sO_{\nP^n})\\
&< &  h^0(\nP^n, \sO_{\nP^n}(2)) + 2 h^0(\nP^n, \sO_{\nP^n}(1)) + 3 h^0(\nP^n, \sO_{\nP^n}) \\
&=&h^0(\nP^{n+2}, \sO_{\nP^{n+2}}(2)).
\end{array}
$$
Thus $H^0(\nP^{n+2}, \sI_{X|\nP^{n+2}}(2)) \neq 0$. Suppose $\ell = n-3$. By Lemma \ref{lem:-K_X=(n-3)He=2}, $a_2=2$ or $3$. We find
$$
\begin{array}{rcl}
h^0(X, \sO_X(3)) &=& h^0(\nP^n, \sO_{\nP^n}(3)) + 2 h^0(\nP^n, \sO_{\nP^n}(2)) + a_2 h^0(\nP^n, \sO_{\nP^n}(1)) + 2h^0(\nP^n, \sO_{\nP^n})\\
&< & h^0(\nP^n, \sO_{\nP^n}(3)) + 2 h^0(\nP^n, \sO_{\nP^n}(2)) + 3 h^0(\nP^n, \sO_{\nP^n}(1)) + 4h^0(\nP^n, \sO_{\nP^n})\\
&=& h^0(\nP^{n+2}, \sO_{\nP^{n+2}}(3)).
\end{array}
$$
This means that $H^0(\nP^{n+2}, \sI_{X|\nP^{n+2}}(3)) \neq 0$. 

\medskip

Suppose that $0 \leq \ell \leq n-4$. As is well-known, there is a rank two vector bundle $F$ on $\nP^{n+2}$ fitting into a short exact sequence
$$
0 \longrightarrow \sO_{\nP^{n+2}} \longrightarrow F \longrightarrow \sI_{X|\nP^{n+2}}(n+3-\ell) \longrightarrow 0.
$$
Note that $c_1(F) = n+3-\ell$ and $c_2(F)=d$. Let $a, b \in \mathbf{C}$ be two roots of $x^2-(n+3-\ell)x+d=0$. Then $a+b = n+3-\ell = c_1(F)$ and $ab=d=c_2(F)$.
We claim that
$$
a=\begin{cases} 2, 3, \ldots, n+1 & \text{ when $n$ is odd} \\
1, 2, \ldots, n+2 & \text{ when $n$ is even}. \end{cases}
$$
Granting the claim, we see that
$$
\chi(F(-3)) =  {a-3+n+2 \choose n+2 } + {b-3 + n+2 \choose n+2 } = {a + n-1 \choose n+2} + {2n+2-a-\ell \choose n+2} > 0.
$$
As $\chi(F(-3))  = h^0(\nP^{n+2}, F(-3)) - h^1(\nP^{n+2}, F(-3))$, we get 
$$
h^0(\nP^{n+2}, \sI_{X|\nP^{n+2}}(n-\ell)) = h^0(\nP^{n+2}, F(-3)) > 0.
$$
To prove the claim, taking a general linear section of dimension $(n-\ell)$, we may assume that $\ell = 0$. We replace $n$ by $n-\ell \geq 4$, so we still have $n \geq 4$.
First, consider the case when $n$ is odd. Note that 
$$
\chi(F(-n-2)) = \chi(\sO_{\nP^{n+2}}(-n-2)) + \chi(\sI_{X|\nP^{n+2}}(1)) =0+0= 0.
$$
We find
$$
\chi(F(-n-2)) = {a \choose n+2} + {n+3-a \choose n+2} = \frac{(2a-(n+3))(a-2) \cdots (a-n-1)}{(n+1)!} = 0,
$$
so $a=2, 3, \ldots, n+1$. Next, consider the case when $n$ is even. Note that
$$
\chi(F(-n-3)) = \chi(\sO_{\nP^{n+2}}(-n-3)) + \chi(\sI_{X|\nP^{n+2}}) = 1 + (-1) = 0.
$$
We find
$$
\chi(F(-n-3)) ={a-1 \choose n+2 } + {n+2-a \choose n+2}= \frac{2(a-1)(a-2) \cdots (a-n-2)}{(n+2)!} = 0,
$$
so $a=1,2, \cdots, n+2$. 
\end{proof}

\begin{proof}[Proof of Theorem \ref{thm:BC}]
Note that $H^i(X, \sO_X(m))=0$ for $1 \leq i \leq n-1$ and $m \in \mathbf{Z}$. Since $X$ is contained in a hypersurface of degree $n$ in $\nP^{n+2}$ by Lemma \ref{lem:h^0(I_X(n))}, we have $\operatorname{length}(\langle P, x \rangle \cap X) \leq n$ for all $x \in X$ and a fixed general point $P \in \nP^{n+2}$. Thus Lemma \ref{lem:projmethod} yields that $X \subseteq \nP^{n+2}$ is a complete intersection. For the last statement, recall from \cite[Theorem 1.2]{KP} that if $d \leq 2n+2$, then $X$ is Fano or Calabi--Yau.
\end{proof}

\section{Proof of Theorem \ref{thm:main}}\label{sec:proof2}
\noindent In this section, we prove Theorem \ref{thm:main} after proving some preliminary results. 

\begin{lemma}[{cf. \cite[Lemma 2.2]{BEL}}]\label{lem:numbgens}
Assume that $n \geq 3$ and $X \subseteq \nP^r$ is arithmetically Cohen--Macaulay and subcanonical with $K_X = k H$ for an integer $k$. If $T$ is  the number of minimal generators of the defining ideal $I_{X|\nP^r}$, then
$$
T \leq (e-1)\left( d-1 -\frac{e}{2} \right) = ed-e - d+1 - \frac{e^2-e}{2}.
$$
\end{lemma}

\begin{proof}
Let $M_H$ be the kernel of the evaluation map $H^0(X, H) \otimes \sO_X \to \sO_X(1)$. Recall that the Koszul cohomology $K_{p,q}(X, H)$ is the complex vector space of $p$-th syzygies of weight $q$ of the section ring $R(X,H)=\bigoplus_{m \in \mathbf{Z}} H^0(X, mH)$ and  $K_{p,q}(X, H)= H^1(X, \wedge^{p+1} M_H \otimes \sO_X(q-1))$ for $p \geq 0$ and $q \geq 1$ (see e.g., \cite[Chapter 2]{AN}). As $X \subseteq \nP^r$ is projectively normal and $\sO_X$ is $(k+n+1)$-regular with respect to $\sO_X(1)$ (see e.g., \cite[Subsection 1.8]{positivity}), we have
$$
T=\kappa_{1,1} + \cdots + \kappa_{1,k+n+1},
$$
where $\kappa_{1,m+1}:=\dim K_{1,m+1}(X, H)$ is the graded Betti number for $0 \leq m \leq k+n$. By Serre duality, $\kappa_{1,k+n+1}=0$, so we only need to consider the cases $0 \leq m \leq k+n-1$. We have
$$
H^2(X, \wedge^2 M_H \otimes \sO_X(m-1)) = H^1(X, M_H \otimes \sO_X(m)) = K_{0,m+1}(X, H)= 0~~\text{ for $m \geq 0$}.
$$
Now, suppose that a general projection $\pi \colon X \to \nP^n$ is given by a subspace $V \subseteq H^0(X, \sO_X(1))$  of dimension $n+1$. If $M_V$ is the kernel of the evaluation map $V \otimes \sO_X \to \sO_X(1)$, then $M_V = \pi^* M_{\sO_{\nP^n}(1)} = \pi^* \Omega_{\nP^n}^1(1)$. Then $\pi_* M_V = \Omega_{\nP^n}^1(1) \otimes \pi_* \sO_X$ and $\pi_* \wedge^2 M_V =\Omega_{\nP^n}^2(2) \otimes \pi_* \sO_X$, where $
\pi_{*}\sO_X = \sO_{\nP^n}^{\oplus a_0} \oplus \sO_{\nP^n}(-1)^{\oplus a_1} \oplus \cdots \oplus \sO_{\nP^n}(-(n+1+k))^{\oplus a_{n+1+k}}$ as in Subsection \ref{subsec:Fano}. By Bott vanishing,
$$
\begin{array}{l}
h^1(X, M_V \otimes \sO_X(m)) = a_{m+1}~~\text{ and }~~h^2(X, M_V \otimes \sO_X(m))=0;\\[3pt]
h^1(X, \wedge^2 M_V \otimes \sO_X(m)) = 0~~\text{ and }~~h^2(X, \wedge^2 M_V \otimes \sO_X(m))=a_{m+2}.
\end{array}
$$
Put $\overline{V}:=H^0(X, \sO_X(1))/V$. Note that $\dim \overline{V} = e$. We have a short exact sequence
$$
0 \longrightarrow M_V \longrightarrow M_H \longrightarrow \overline{V} \otimes \sO_X \longrightarrow 0.
$$
There exists a vector bundle $K$ on $X$ fitting in short exact sequences
$$
0 \longrightarrow \wedge^2 M_V \longrightarrow \wedge^2 M_H \longrightarrow K \longrightarrow 0~~\text{ and }~~0 \longrightarrow M_V \otimes \overline{V} \longrightarrow K \longrightarrow \wedge^2 \overline{V} \otimes \sO_X \longrightarrow 0.
$$
We have
$$
\kappa_{1,m+1} = h^1(X, \wedge^2 M_H \otimes \sO_X(m)) = h^1(X, K \otimes \sO_X(m)) - h^2(X, \wedge^2 M_V \otimes \sO_X(m)),
$$
and
$$
\begin{array}{rcl}
h^1(X, K \otimes \sO_X(m)) &=& \dim \overline{V} \cdot h^1(X, M_V \otimes \sO_X(m)) \\
&&- \operatorname{corank}\big(H^0(X, K \otimes \sO_X(m)) \longrightarrow \wedge^2 \overline{V} \otimes H^0(X, \sO_X(m))\big).
\end{array}
$$
When $m = 0$, since $H^0(X, K)=0$, we get $\kappa_{1,1} = ea_1 - a_2 - {e \choose 2}$.
When $1 \leq m \leq k+n-1$, we get $\kappa_{1,m+1} \leq ea_{m+1} - a_{m+2}$.
Then we find
$$
T \leq e(a_1 + \cdots + a_{k+n}) - (a_2 + \cdots + a_{k+n+1}) - {e \choose 2} =   e(d-2) - (d-e-1) - {e \choose 2},
$$
which completes the proof.
\end{proof}

The proof of the lemma also yields the following.

\begin{corollary}
Suppose that $X$ is cut out scheme-theoretically in $\nP^r$ by hypersurfaces of degree at most $d_0 \geq 3$. If
$$
f(e):= (e-1) \sum_{i=2}^{d_0-1} {e+i-1 \choose i} + \frac{e^2+e}{2} \leq n+1,
$$
then $X \subseteq \nP^r$ is a complete intersection.
\end{corollary}

\begin{proof}
Notice that $f(e)$ is a polynomial of degree $d_0$ in $e$ and $f(e) \geq e(d_0-1)+1$. As $f(e) \leq n+1$, \cite[Corollary 2]{BEL} shows that $X \subseteq \nP^r$ is arithmetically Cohen--Macaulay. Note that $f(e) \geq e^2(e+1)/2 \geq 3e$. Thus $n \geq 3e-1 \geq e+3$. By Barth--Larsen theorem (cf. \cite[Corollary 3.2.3]{positivity}), $K_X = kH$ for some integer $k$. We may assume that $3 \leq d_0 \leq k+n+1$. Notice that
$$
a_i \leq {e+i-1 \choose i}~~\text{ for all $0 \leq i \leq n+1+k$}.
$$
If $T$ is the minimal number of hypersurfaces in $\nP^r$ cutting out $X$ scheme-theoretically, then
$$
\begin{array}{l}
\displaystyle T \leq \kappa_{1,1} + \cdots + \kappa_{1,d_0-1} \leq e(a_1 + \cdots + a_{d_0-1}) - (a_2 + \cdots + a_{d_0}) - {e \choose 2}\\
\displaystyle  \leq (e-1)(a_2 + \cdots + a_{d_0-1}) + \frac{e^2+e}{2} \leq (e-1) \sum_{i=2}^{d_0-1} {e+i-1 \choose i} + \frac{e^2+e}{2}=f(e) \leq n+1.
\end{array}
$$
Now, \cite[Theorem 3.2]{Netsvetaev} implies that $X \subseteq \nP^r$ is a complete intersection.
\end{proof}

\begin{remark}
When $d_0=2$ in the situation of the corollary, Ionescu--Russo \cite[Theorem 3.8]{IR2} proved that if $n \geq 2e+1$, then $X \subseteq \nP^r$ is a complete intersection.
\end{remark}

\begin{theorem}[{cf. \cite[Corollary 3]{BEL}}]\label{thm:(n+2)/e+2}
If
$$
d \leq \frac{n+2}{e}+e, 
$$
then $X \subseteq \nP^r$ is a complete intersection unless it is projectively equivalent to a variety of minimal degree with $n-1 \leq e \leq n+2$ or the Pl\"{u}cker embedding $\mathbf{G}(1,4) \subseteq \nP^9$ or its linear section with $4 \leq n \leq 6$. 
\end{theorem}

\begin{proof}
We only need to consider the case that $e \geq 2$.
First, suppose that $e=2$. If $n \leq 3$, then $d \leq 4$ and $X \subseteq \nP^r$ is a complete intersection of type $(2,2)$ or a variety of minimal degree. If $n \geq 4$, then the assertion follows from Theorem \ref{thm:BC}. Next, suppose that $e \geq 3$. If $d=e+1$, then $e \leq n+2$ and clearly $e \geq n-1$. If $d =e+2$, then $2e \leq n+2$ so that $n \geq 4$. By Fujita's classification \cite{Fujita1, Fujita2}, if $X \subseteq \nP^r$ is not a complete intersection, then it is the Pl\"{u}cker embedding $\mathbf{G}(1,4) \subseteq \nP^9$ or its linear section. 
Now, we assume that $d \geq e+3$. As $e+3 \leq (n+2)/e + e$, we get $n \geq 3e-2 \geq 2e+1$. Since $d \leq ne+1$, Barth--Larsen theorem (cf. \cite[Corollary 3.2.3]{positivity}) and \cite[Theorem 1.2]{KP} show that $X$ is a Fano variety with $-K_X = \ell H$ for a positive integer $\ell$. By \cite[Theorem 9]{Noma}, $X \subseteq \nP^r$ is $k$-normal for $k \geq e(d-e+1)-(n+e)$, but the given condition says $e(d-e+1)-(n+e) \leq 2$. Together with Zak's theorem on linear normality (cf. \cite[Theorem 3.4.25]{positivity}), we see that $X \subseteq \nP^r$ is projectively normal. Thus $X \subseteq \nP^r$ is arithmetically Cohen--Macaulay. Now, applying Lemma \ref{lem:d<=e(e-1)/2+2} and the given condition, we find
$$
ed-e -d+1 - \frac{e^2-e}{2} \leq ed-e - \frac{e^2}{2}-1 - \frac{e^2-e}{2} = ed-e^2-1- \frac{e}{2} \leq ed -e^2-2 \leq n.
$$
Then Lemma \ref{lem:numbgens} says that $I_{X|\nP^r}$ can be generated by $n$ generators. As $n \geq 3e-2$, \cite[Theorem 3.2]{Netsvetaev} implies that $X \subseteq \nP^r$ is a complete intersection.
\end{proof}

We are ready to prove Theorem \ref{thm:main}.

\begin{proof}[Proof of Theorem \ref{thm:main}]
By Lemma \ref{lem:d<=e(e-1)/2+2} and the given condition,
$$
\frac{e^2}{2} \leq d-2 \leq   \sqrt[3]{\frac{(n+e^2-2e+2)^2}{2}},
$$
so 
$$
\frac{e^3}{2} \leq n+e^2-2e+2.
$$
This implies that
$$
d \leq  \sqrt[3]{\frac{(n+e^2-2e+2)^2}{2}}+2 \leq \frac{n+e^2-2e+2}{e}+2 = \frac{n+2}{e}+e.
$$
Now, the theorem follows from Theorem \ref{thm:(n+2)/e+2} since in our case $X \subseteq \nP^r$ is neither a variety of minimal degree nor the Pl\"{u}cker embedding $\mathbf{G}(1,4) \subseteq \nP^9$ or its linear section.
\end{proof}


\end{document}